\documentclass[a4paper, 12pt]{article}
\usepackage{authblk}
\usepackage{hyperref}

\title{Quantum $K$-theory of Lagrangian Grassmannian via parabolic Peterson isomorphism% Title
\footnotetext{Key words and phrases: Schubert calculus, Lagrangian Grassmannian, Quantum $K$-theory, affine Grassmannian}% Key words
\footnotetext{Mathematics Subject Classification 2020: Primary 14M15; Secondary 53D45.}% MSC classification
} 
\author[1]{Takeshi Ikeda\footnote{e-mail: \url{gakuikeda@waseda.jp}}}
\affil[1]{
Department of Pure and Applied Mathematics,
Faculty of Science and Engineering, 
Waseda University
 3-4-1 Okubo, Shinjuku-ku, Tokyo 169-8555, Japan
}
\author[1]{Takafumi Kouno\footnote{e-mail: \url{t.kouno@aoni.waseda.jp}}}
\author[1]{Yusuke Nakayama\footnote{e-mail: \url{yusuke216144@akane.waseda.jp}}}
\author[2]{\\ Kohei Yamaguchi\footnote{e-mail: \url{yamaguchi.kohei@math.nagoya-u.ac.jp}}}
\affil[2]{
Graduate School of Mathematics, Nagoya University, Furo-cho, Chikusa-ku, Nagoya 464-8602, Japan
}

\date{}

\usepackage[margin=30truemm]{geometry}

\usepackage[mathscr]{eucal}

\usepackage{amsmath, amssymb}
\allowdisplaybreaks[1]
\usepackage{amsthm}
\usepackage{mathtools}
\mathtoolsset{showonlyrefs}
\usepackage{enumerate}
\usepackage{ytableau}
\ytableausetup{boxsize=0.4em}

\numberwithin{equation}{section}

\theoremstyle{plain}
\newtheorem{thm}{Theorem}[section]
\newtheorem{prop}[thm]{Proposition}
\newtheorem{lem}[thm]{Lemma}
\newtheorem{cor}[thm]{Corollary}

\theoremstyle{remark}
\newtheorem{rem}[thm]{Remark}

\newcommand{\alv}{\alpha^\vee}
\newcommand{\BZ}{\mathbb{Z}}
\newcommand{\BC}{\mathbb{C}}

\newcommand{\CO}{\mathcal{O}}
\newcommand{\SP}{\mathscr{SP}}
\newcommand{\eps}{\varepsilon}
\newcommand{\vpi}{\varpi}
\newcommand{\vpiv}{\varpi^{\vee}}

\newcommand{\Gr}{\mathrm{Gr}}
\newcommand{\GrG}{\Gr_{G}}
\newcommand{\GrC}{\Gr_{C_n}}
\newcommand{\la}{\lambda}
\newcommand{\JLG}{J_{\LG}}
\newcommand{\JLGh}{J_{\LG}^{\mathrm{hom}}}

\newcommand{\KGr}{K_{\ast}^{T}(\GrG)}
\newcommand{\KGrC}{K_{\ast}^{T}(\GrC)} 
\newcommand{\Sp}{\mathrm{Sp}_{2n}}
\newcommand{\LG}{\mathrm{LG}(n)}
\newcommand{\LGrm}{\mathrm{LG}} 
\newcommand{\PC}{\mathscr{P}_{C}^{n}} 
\newcommand{\Pet}{{\Phi_{P}}}

\newcommand{\af}{\mathrm{af}}
\newcommand{\poly}{\mathrm{poly}}
\newcommand{\loc}{\mathrm{loc}}

\newcommand{\mcr}[1]{\lfloor #1 \rfloor}
\newcommand{\bra}[1]{[\![ #1 ]\!]}

\DeclareMathOperator{\kernel}{Ker}

\newcommand{\map}[5][\rightarrow]{%
\begin{gathered} #2 \\ #4 \end{gathered} \ %
\begin{gathered} #1 \\ \mapsto \end{gathered} \ %
\begin{gathered} #3 \\ #5 \end{gathered}}

\newcommand{\mapname}[1]{%
\begin{gathered} #1 : \\ \ \end{gathered} \ }

\begin{document}

\maketitle

\begin{abstract}
We study Schubert calculus in the torus-equivariant quantum $K$-ring of the Lagrangian Grassmannian $\LG$. Our main tool is the $K$-theoretic Peterson map due to Kato.
The map is from the (localized) equivariant $K$-homology ring $K_*^T(\Gr_G)$ of the affine Grassmannian $\Gr_G$ of the symplectic group $G=\Sp(\BC)$ to the (localized)  torus-equivariant quantum $K$-ring $QK_T(\LG).$
We determine explicitly the kernel of this map. 
\end{abstract}

%=================%
% START SECTION 1 %
%=================%
\section{Introduction}
Let $\LG$ be the \emph{Lagrangian Grassmannian} of maximal isotropic subspaces of $\BC^{2n}$ equipped with a non-degenerate skew-symmetric bilinear form. The space 
$\LG$ can be realized as a homogeneous space $G/P$ of the symplectic group $G={\Sp}
(\BC)$ where $P$ is the maximal 
parabolic subgroup corresponding to the long simple 
root $\alpha_n$.
Let $T$ be the maximal torus of $G$ contained in $P.$
In this paper, we study the $T$-equivariant (small) quantum $K$-ring
$QK_T(\LG).$ 

Let $R(T)$ be the representation ring of $T$, and $Q$ be an indeterminate called the Novikov variable.
The $T$-equivariant quantum $K$-ring 
$QK_T(\LG)$ (see \S\ref{sec:parabolic_K-Peterson})
is a commutative
$R(T)[Q]$-algebra
with 
$R(T)[Q]$-free basis $\{\CO^\la_{\LGrm}\}$
consisting of 
the classes of structure sheaves of the Schubert varieties in 
$\LG$
indexed by the set $\SP(n)$ of the $n$-bounded 
strict partitions $\la$. 
We are interested in 
the Schubert calculus of 
$QK_T(\LG)$, i.e., the combinatorial study of 
the multiplicative structure constants 
of $QK_T(\LG)$ with respect to the Schubert basis $\{\CO^\la_{\LGrm}\}$. 

The aim of this paper is to relate the  ring $QK_T(\LG)$ with the $T$-equivariant $K$-homology ring 
$K_*^T(\Gr_G)$ of the affine Grassmannian of the symplectic group.
$K_*^T(\Gr_G)$ has a Schubert basis
in such a way that Schubert calculi 
of those rings are also nicely related.
We apply ``quantum equals affine'' phenomena 
to $QK_T(\LG)$. The principle was established by Lam--Shimozono \cite{LS} for the $T$-equivariant quantum cohomology $QH_T^*(G/B)$ of the flag variety $G/B$ and the $T$-equivariant homology ring $H_*^T(\GrG)$ of the affine Grassmannian of $G.$ In fact, the so-called the Peterson isomorphism give an isomorphism between these rings up to certain 
localizations.
The Peterson isomorphism in the $K$-theory setting was studied in \cite{IIM} for type $A$ flag variety in non-equivariant case. In \cite{LLMS} a conjecture for a $K$-Peterson isomorphism was formulated in general. The conjecture was affirmatively solved by Kato \cite{K1}, and by Chow--Leung \cite{CL} by a different method.
In \cite{K1} and \cite{CL}, they also 
provide a \textit{parabolic} version of $K$-theoretic 
Peterson map that works for general $G/P.$
In the parabolic case, 
the map 
is not injective in general, although it is surjective. 
The kernel of the map is not known explicitly in general. The main result of this paper is to give an explicit description of the kernel in the case of $G/P=\LG.$

The Schubert varieties of $\GrG$ for $G=\Sp(\BC)$ is 
indexed by 
the set $\PC$ defined by 
\begin{equation}
\mathscr{P}^n_C := \left\{ \lambda = (\lambda_{1}, \ldots, \lambda_{l}) \ \middle| \ \begin{gathered} 2n \ge \lambda_{1} \ge \cdots \ge \lambda_{l} > 0 \\ \lambda_{k} < n \ \Rightarrow \ \lambda_{k} > \lambda_{k+1} \end{gathered} \right\}. \label{eq:PC}
\end{equation}
In fact, the set $\PC$ is
in bijection with the set $W_\af^0$ of 
minimal length coset representative for $W_\af/W$, where $W_\af$ and $W$ are the affine Weyl group and the Weyl group $G$ respectively.

One notes that 
$\SP(n)$ is naturally embedded in 
$\mathscr{P}^n_C$ as the subset 
consisting of $\la\in \mathscr{P}^n_C$
such that $\la_1\le n.$
This observation is basic in our study. 
Define the ideal 
\begin{equation} \label{eq:def_JLG}
\JLG := (\CO_{\lambda}^\Gr - \CO_{(n+1)}^\Gr\CO_{\lambda^{-}}^\Gr \mid \lambda \in \PC, \ \lambda_{1} \ge n+1) 
\end{equation}
of $K_*^T(\GrC),$
where $\la^{-}$
is the partition obtained from $\la$ by removing the first row of $\la,$ i.e., $\la^{-}=(\la_2,\ldots,\la_l)$ when $\la = (\la_{1}, \ldots, \la_{l})$. 
We denote by $(K_*^T(\GrC)/\JLG)_\loc$ the localization of the quotient ring $K_*^T(\GrC)/\JLG$ by the multiplicative system generated by
the image of $\CO_{(n+1)}^{\Gr}$.
We regard the localized ring $(\KGrC / \JLG)_\loc$
as an $R(T)[Q,Q^{-1}]$-algebra so that $Q$
acts by multiplication of $(\CO_{(n+1)}^{\Gr})^{-1}.$
For $\lambda \in \PC$, we denote by $\lambda^{\le n} \in \SP(n)$ the partition obtained from $\lambda$ by removing all parts of $\lambda \in \PC$ greater than $n$. 
For $\lambda = (\lambda_{1}, \ldots, \lambda_{l}) \in \SP(n)$, set $\lambda^* := (n+1-\lambda_{l}, \ldots, n+1-\lambda_{1}) \in \SP(n)$. 

\begin{thm} \label{thm:A}
There exists an isomorphism of $R(T)[Q,Q^{-1}]$-algebras \label{thm:parabolic_Peterson_intro}
\begin{align}
\left(
K_*^T(\GrC)/\JLG
\right)_\loc
&\rightarrow
QK^\poly_T(\LG)_\loc,\\
\CO^\Gr_\la
\left(\CO^\Gr_{(n+1)}\right)^{-k}
&\mapsto Q^{k-\ell(\la)}\CO^{(\la^{\le n})^*}_{\LG}
\end{align}
where $k\in \BZ$, and $\ell(\lambda)$ is the length of $\lambda \in \PC$.
\end{thm}

Let us show an example of applying the isomorphism of the theorem. 
Unfortunately, 
explicit calculations
in the Schubert classes in 
$K_*^T(\GrC)$
are not much available in the existing literature. So
let us start with the Chevalley formula (\cite[Theorem~49]{LNS}) for $QK_{T}(\mathrm{Sp}_{4}(\BC)/B)$, which 
is translated into the following equation
by the $K$-Peterson isomorphism
    \begin{equation}
        \CO_{\ydiagram{3,2}}^{\Gr} \CO_{\ydiagram{2,1}}^{\Gr} = (1-e^{-2(\alpha_{1}+\alpha_{2})}) \CO_{\ydiagram{3,3,2,1}}^{\Gr} + e^{-2(\alpha_{1}+\alpha_{2})}\CO_{\ydiagram{3,3,2}}^{\Gr},
    \end{equation}
    where $\alpha_1,\alpha_2$ are the simple roots of type $C_2$ so that $\alpha_2$ is long.
    We have the following congruence 
    with respect to $J_{\LGrm(2)}$:
    \begin{equation}
    \CO_{\ydiagram{3,2}}^{\Gr}\equiv
    \CO_{\ydiagram{3}}^{\Gr}  \CO_{\ydiagram{2}}^{\Gr},\quad
    \CO_{\ydiagram{3,3,2,1}}^{\Gr}\equiv
      (\CO_{\ydiagram{3}}^{\Gr})^2
     \CO_{\ydiagram{2,1}}^{\Gr},\quad 
    \CO_{\ydiagram{3,3,2}}^{\Gr}\equiv
      (\CO_{\ydiagram{3}}^{\Gr})^2
     \CO_{\ydiagram{2}}^{\Gr}.
    \end{equation}
    By applying the isomorphism in Theorem~\ref{thm:A}, we obtain 
    \begin{equation}
        \begin{split}
            & Q^{-1} \left( Q^{-1}\CO_{\LGrm}^{\ydiagram{1}} \right) \left( Q^{-2}\CO_{\LGrm}^{\ydiagram{2,1+1}} \right) \\ 
            &= (1-e^{-2(\alpha_{1}+\alpha_{2})}) Q^{-2} \left( Q^{-2}\CO_{\LGrm}^{\ydiagram{2,1+1}} \right) + e^{-2(\alpha_{1}+\alpha_{2})} Q^{-2} \left( Q^{-1}\CO_{\LGrm}^{\ydiagram{1}} \right), 
        \end{split}
    \end{equation}
    in $QK_{T}(\LGrm(2))_{\loc}$, where
    we use shifted Young diagrams for
    $\lambda\in \SP(2).$ Thus we have 
    \begin{equation}
        \CO_{\LGrm}^{\ydiagram{1}} \CO_{\LGrm}^{\ydiagram{2,1+1}} = (1-e^{-2(\alpha_{1}+\alpha_{2})}) \CO_{\LGrm}^{\ydiagram{2,1+1}} + e^{-2(\alpha_{1}+\alpha_{2})} Q\, \CO_{\LGrm}^{\ydiagram{1}}. \label{eq:LG_product}
    \end{equation}
    This is the Chevalley formula in $QK_{T}(\LGrm(2))$ described in \cite{BCMP} or \cite{KLNS}. 

Note that in $K_*^T(\Gr_{C_{4}})$, we have
  \begin{align}
        \CO_{\ydiagram{2}}^{\Gr} \CO_{\ydiagram{2,1}}^{\Gr} &=e^{-2(\alpha_{1}+\alpha_{2})} \CO_{\ydiagram{3,2}}^{\Gr}
        +(e^{-\alpha_1-\alpha_2}+e^{-\alpha_2})(1-e^{-\alpha_1-\alpha_2})
        \CO_{\ydiagram{3,2,1}}^{\Gr} 
        \\
        &+(1-e^{-\alpha_2})(1-e^{-\alpha_{1}-\alpha_{2}})  \CO_{\ydiagram{4,2,1}}^{\Gr} .
    \end{align}
    This can be directly verified by using 
    results in \cite{IISY}.
    One sees that 
    this equation leads to \eqref{eq:LG_product}.

It is reasonable to 
state the 
corresponding
result in homology case.
Let $H_*^T(\GrC)$ 
and $QH^*_{T}(\LG)$ be
the equivariant homology ring
of $\GrC$ and 
the equivariant quantum cohomology ring of $\LG$ respectively.
$QH^*_{T}(\LG)$ 
is a commutative $H_T^*(pt)[q]$-algebra with a quantum parameter $q$.
For $\la\in \PC$, we denote by $\sigma_{\la}^{\Gr}$ be the 
Schubert class in  $H_*^T(\GrC)$. 
For $\la\in \SP(n)$, we denote by $\sigma^{\la}_{\LGrm}$ be the 
Schubert class in  $QH_T^*(\LG)$. 
Let 
$\JLGh$ be the $S$-span of 
$\sigma_\lambda^\Gr$ such that $\lambda_1\ge n+2$
in $H_*^T(\GrC)$, which  is known to be 
an ideal 
of $H_*^T(\GrC).$
We denote by $(H_*^T(\GrC)/\JLGh)_\loc$ the localization of $H_*^T(\GrC)/\JLGh$ by the multiplicative system generated by
the image of $\sigma_{(n+1)}^{\Gr}$. 
We define $QH^*_{T}(\LG)_{\loc}
=QH^*_{T}(\LG)[q^{-1}].$
The following is an immediate consequence of Theorem \ref{thm:A}. See \S \ref{sec:quantum_cohomology} for details.
\begin{cor}[cf. \cite{LS}] \label{cor:homology}
There exists an isomorphism of $H_T^*(pt)[q,q^{-1}]$-algebras 
\begin{align}
(H_*^T({\GrC})/\JLGh)_\loc
&\rightarrow
QH^*_{T}(\LG)_{\loc},\\
\sigma_{\lambda}^{\Gr} 
(\sigma_{(n+1)}^{\Gr})^{-k}&\mapsto\begin{cases}
    q^{k-\ell(\lambda)}
    \sigma_{\LGrm}^{(\lambda^{\le n})^*} & \text{if $\la_1=n+1$}\\
    0 & \text{if $\la_1\ge n+2$}
\end{cases}.\label{eq:homology-map}
\end{align}
\end{cor}
Note that 
the isomorphism is 
degree reversing with 
$\deg(\sigma_\la^\Gr)=|\la|$ for affine side, and $\deg(q)=n+1,\;
\deg(\sigma_{\mathrm{LG}}^\la)=|\la|$
for quantum side. We also remark that a presentation for $QH^*_{T}(\LG)$ in terms of the factorial $Q$-functions 
was obtained in \cite{IMN} (we do not know the exact connection between the two presentations). Actually in the early stage of this project, we tried to generalize \cite{IMN} to $K$-theory setting using the (factorial) $GQ$-functions (\cite{IN}) so that $QK(\LG)$ is 
presented as a quotient ring of 
the ring of $K$-theoretic $Q$-functions. 
It turns out that the complete description of the ideal of relations is quite hard to get.

According to Peterson \cite{Pet} (see also \cite{LS}), there is an ideal $J_{P}^{\mathrm{Pet}} \subset H_{\ast}^{T}(\GrC)$ such that $(H_{\ast}^{T}(\GrC)/J_{P}^{\mathrm{Pet}})_{\loc} \simeq QH_{T}^{\ast}(\LG)_{\loc}$, which is defined in terms of a subset $(W^P)_\af$ of
of $W_\af$ called the ``Peterson's coset representatives'' of $W_{\af}$ (see \S \ref{ssec:Pet}). 
We can deduce that the subset $\{\la\in \PC\mid \la_1\le n+1\}$ of $\PC$ is in bijection with the set $W_\af^0\cap (W^P)_\af$ (Corollary \ref{cor:Petco}).

Now that we have Theorem \ref{thm:A}, it would be desirable to develop combinatorics of $K_*^T(\GrC)$ further. 
In relation with this problem, there is
the dual $P$-Schur functions $gp_\la(y)$ indexed by the strict partitions $\la$
 introduced by Nakagawa and Naruse \cite{NN}. 
 The set of functions $\{gp_\la\}$ is defined as the dual basis to the $GQ$-functions $\{GQ_\la\}$ introduced by Ikeda--Naruse \cite{IN} that is identified with the Schubert basis of the $K$-ring of the Lagrangian Grassmannian $\mathrm{LG}(\infty)$ of infinite rank.  
 It is notable that further properties of the function $gp_\la$  have been intensively studied by Lewis--Marberg \cite{LM}, and Iwao \cite{Iw}. 
It is natural to expect that there is the affine (double) version of the $gp$-functions indexed by $\la\in\PC$ such that it is identified with the Schubert class of $K_*^T(\GrC).$ 
These issues will be pursued in upcoming work \cite{IISY}. 

It is natural to study the quantum $K$-ring of other homogeneous 
variety via the parabolic Peterson isomorphism. Our particular interest in this regard is the cominuscule $G/P.$
In fact, Peterson, Lam--Shimozono showed the equivariant quantum cohomology ring of cominuscule $G/P$ is nicely related to 
the equivariant homology of the affine Grassmannian $\GrG$ (\cite[Theorem 10.21, Proposition 11.3]{LS}). 
For instance, Cookmayer--Mili\'cevi\'c \cite{CM}
studied the quantum cohomology of the ordinary Grassmannian in this context.

This paper is organized as follows. 
In Section~\ref{sec:parabolic_K-Peterson}, we review the parabolic $K$-Peterson isormophism. 
In Section~\ref{sec:0-Grassmannian}, we describe some required combinatorial facts for the set of
affine Grassmannian elements in type $C$. 
In Section~\ref{sec:proof}, we give a proof of Theorem~\ref{thm:A}, which is the main result of this paper. 
In Section~\ref{sec:quantum_cohomology}, we compare our $K$-Peterson isomorphism with the isomorphism for (co)homology rings established by Peterson and Lam--Shimozono
and describe the set of Peterson's coset representatives in terms of partitions.

\subsection*{Acknowledgments}
We are grateful to Satoshi Naito for helpful discussion. 
We used "Equivariant Schubert Calculator" (\url{https://sites.math.rutgers.edu/~asbuch/equivcalc/}), developed by A.S.~Buch, to compute examples. 
T.K. was partly supported by JSPS Grants-in-Aid for Scientific Research 20J12058, 22J00874, and 22KJ2908. 
K.Y. is supported by JSPS Fellowships for Young Scientists No. 22J11816.
T.I. was partly supported by JSPS Grants-in-Aid for Scientific Research  23H01075,
22K03239,
20K03571,
20H00119.

%=================%
% START SECTION 2 %
%=================%

\section{\texorpdfstring{Parabolic $K$-Peterson isomorphism}{Parabolic K-Peterson isomorphism}} \label{sec:parabolic_K-Peterson}

In this section, we review 
the parabolic Peterson homomorphism in $K$-theory due to 
Kato \cite{K1} and Chow--Leung \cite{CL}.
We refer to \cite{LSS} for the basic properties of the $T$-equivarinat $K$-theory
of the affine Grassmannian.

Let $G$ be a connected
simply-connected linear algebraic group over $\BC$. We fix a maximal torus $T$ of $G$ and a Borel subgroup $B$ containing $T$. Let $P$ be a parabolic subgroup $G$ containing $B$. The coset space $G/P$ has a structure of non-singular 
projective variety.
We denote by $R(T)$ the representation ring of $T$. 

Let $I$ be the index set of the Dynkin diagram of $G$, and $J$ be the subset 
corresponding to the parabolic subgroup $P.$ 
Let $\{\alv_i\mid i \in I\}$ be the set of simple coroots, $Q^{\vee} = \sum_{i \in I} \BZ \alv_i$ the coroot lattice of $G$, $Q^{\vee, +} = \sum_{i \in I} \BZ_{\ge 0} \alpha_{i}^{\vee}$ the positive part of $Q^{\vee}$. 
Let $J$ be the subset of $I$ corresponding to $P$ and set $Q_{P}^{\vee} := \sum_{i \in J} \BZ \alpha_{i}^{\vee}$, $Q_{P}^{\vee, +} := \sum_{i \in J} \BZ_{\ge 0} \alpha_{i}^{\vee}$. For each $\xi = \sum_{i \in I} c_{i} \alpha_{i}^{\vee} \in Q^{\vee}$, set $[\xi] := \sum_{i \in I \setminus J} c_{i} \alpha_{i}^{\vee}$. Note that $[\xi]$ is a representative of the coset $\xi + Q_{P}^{\vee}$ in the quotient $Q^{\vee}/Q_{P}^{\vee}$. 

Let $W$ be the Weyl group of $(G,T)$ with 
the simple reflections $s_{i}$, $i \in I$, as generators. 
Define the subgroup $W_{P}:= \langle s_{i} \mid i \in J \rangle$ of $W$. 
Let $W^P$ be the minimal length coset representatives for $W/W_P.$
For $w \in W$, we denote by $\mcr{w}$ the element of $W^P$ corresponding to  $wW_P\in W/W_{P}$.  
We set $R(T)\bra{Q}= R(T)\bra{Q_{i} \mid i \in I \setminus J}$. 
The $T$-equivariant quantum $K$-ring $QK_{T}(G/P)$ is defined as an $R(T)$-module by 
\begin{equation}
    QK_{T}(G/P) := R(T)\bra{Q} \otimes_{R(T)}K_{T}(G/P) , 
\end{equation}
where $K_{T}(G/P)$ denotes the Grothendieck groups of the $T$-equivariant vector bundles on $G/P$. 
There is a ring structure of $QK_{T}(G/P)$ defined by Givental \cite{Giv} and Lee \cite{Lee}.

Let $B^{-}$ be a Borel subgroup of $G$ such that $B\cap B_{-}=T.$
For $w \in W^P$, we define the  \emph{Schubert variety} $X^{w}$, which is the closure of the $B^{-}$-orbit of the $T$-fixed point $G/P$ corresponding to $w \in W^{P}$. 
We denote by $\CO^{w}_{G/P}$, $w \in W$, the structure sheaf of $X^{w}$. 
We denote by $[\CO^{w}_{G/P}] \in K_{T}(G/P)$ the class of the sheaf $\CO^{w}_{G/P}$ and we call it the \emph{Schubert class} corresponding to $w$. 
By abuse of notation, we simply
denote $[\CO^{w}_{G/P}]$ by 
$\CO^{w}_{G/P}.$ 
Then $\{\CO^{w}_{G/P} \mid w \in W^{P}\}$ forms an $R(T)\bra{Q}$-basis of $QK_{T}(G/P)$. 
According to Anderson-Chen-Tseng \cite{ACT}, Buch-Chaput-Mihalcea-Perrin \cite{BCMP}, Kato \cite{K1}, and Chow-Leung \cite{CL}, 
the $R(T)[Q]$-span $QK_{T}^{\poly}(G/P)$ of $\CO^w_{G/P}\;(w\in W^P)$ 
is a subring of $QK_{T}(G/P)$. 

Let $W_{\af} $ the affine Weyl group of $G$ with (affine) simple reflections $s_i$ for $i\in I \sqcup \{0\}.$
We have 
$W_\af\cong W \ltimes Q^{\vee}$
$= \{ wt_{\xi} \mid w \in W, \ \xi \in Q^{\vee} \}$, where $t_{\xi}$, $\xi \in Q^{\vee}$, denotes the translation element corresponding to $\xi$. 
The $T$-equivariant $K$-homology 
$K_*^T(\Gr_G)$ of the affine Grassmannian $\Gr_G$ is a commutative  $R(T)$-algebra with an $R(T)$-basis $\{\CO_x^{\Gr}\mid x\in W_\af^0\}$ called the Schubert 
basis, where the index set of the Schubert basis of 
$K_*^T(\Gr_G)$ is the set of minimal length coset representatives $W_\af^0$ for
$W_\af/W$. 

We denote by $Q_{< 0}^{\vee}$ the set of all strictly anti-dominant elements in $Q^{\vee}$. Note that for any $\xi\in Q_{< 0}^{\vee}$ and $w\in W,$
we have $wt_\xi\in W_\af^0$; see \cite[Section~3]{LS}. 
 Let $\KGr_\loc$ be the localization of $\KGr$ by the multiplicative set $\{\CO^\Gr_{t_\gamma}\mid \gamma \in Q_{<0}^\vee\},$
where $\CO^\Gr_{t_\gamma}$ denotes the Schubert class
corresponding to $t_\gamma W\in W_\af/W\cong W_\af^0.$
We also set 
\begin{equation}
QK_{T}^{\poly}(G/P)_\loc:=QK_{T}^{\poly}(G/P)[Q_{i}^{-1} \mid i \in I \setminus J].
\end{equation}
For each $\xi = \sum_{i \in I \setminus J} c_{i}\alpha_{i}^{\vee} \in Q_{P}^{\vee}$, we define $Q^{\xi} := \prod_{i \in I \setminus J} Q_{i}^{c_{i}}$. 

\begin{thm}[{\cite[Theorem~2.22]{K2}}, {\cite[Theorem~1.1]{CL}}]\label{thm:parabolic_peterson_surj}
There is a surjective $R(T)$-algebra homomorphism
\begin{align}
\mapname{\Pet} \map[\twoheadrightarrow]{\KGr_\loc}{QK_{T}^{\poly}(G/P)_\loc,}{\CO_{wt_{\xi}}^\Gr(\CO_{t_{\gamma}}^\Gr)^{-1}}{Q^{[\xi - \gamma]} \CO_{G/P}^{\mcr{w}},}
\end{align}
 for $w \in W$ and $\xi \in Q^{\vee}$ with $wt_{\xi} \in W_{\af}^{0}$, and $\gamma \in Q_{<0}^\vee.$
\end{thm}

Note that there is a homology version of parabolic 
Peterson homomorpshim $\Phi_P^{\mathrm{hom}}$ by Lam--Shimozono \cite{LS}, and
$\kernel(\Phi_{P}^{\mathrm{hom}})$ can be 
explicitly described in terms of ``Peterson coset representatives'' (\cite[Section~10.4]{LS}).
However, no explicit description 
of $\kernel(\Pet)$ is known for general $G/P$.

%=================%
% START SECTION 3 %
%=================%

\section{\texorpdfstring{Preliminaries of on $W_\af^0$ in type $C$}{Preliminaries of on Waf0 in type C}} \label{sec:0-Grassmannian}
In this section, we prove some results on the $0$-Grassmannian elements in the type $C$ affine Weyl group. 
 
 \subsection{\texorpdfstring{Combinatorial description of $W_\af^0$}{Combinatorial description of Waf0}}
Let 
 $W_\af=\langle s_0,s_1,\ldots,s_n\rangle$ be the affine Weyl group of type $C_n^{(1)}.$ The defining relations are given as follows:
 \begin{alignat}{2}
 s_i^2&=1 & \quad &\text{for $0\le i\le n$;}
 \\ 
 s_0s_1s_0s_1&=s_1s_0s_1s_0; & & \\
 s_{n-1}s_ns_{n-1}s_n&=s_ns_{n-1}s_ns_{n-1}; & & \\
s_is_{i+1}s_i&=s_{i+1}s_{i}s_{i+1} & &\text{for $1\le i\le n-2$}; \\
s_is_j&=s_js_i & &\text{for $0\le i,j\le n$ such that $|i-j|\ge 2$.}
 \end{alignat}
 Let $\ell: W_\af\to \BZ_{\ge 0}$ be the length function. The subgroup
 $W=\langle s_1,\ldots,s_n\rangle$
 is the Weyl group of type $C_n.$
 Define
 \begin{equation}
W_\af^0 :=\{w\in W_\af\;|\; 
\ell(ws_i)=\ell(w)+1\quad 
\text{for $1\le i\le n$}\}. 
 \end{equation}
It is known that the set $W_\af^0$ is a 
complete representative for the coset space $W_\af/W.$
 For each $i$ such that $1 \le i \le 2n$, there is an element 
 $\rho_{i} \in W_{\af}^0$ defined by 
\begin{equation}
\rho_{i} := \begin{cases} s_{i-1} \cdots s_{1}s_{0} & \text{if } 1 \le i \le n, \\ s_{2n-i+1} \cdots s_{n-1}s_{n}s_{n-1} \cdots s_{1}s_{0} & \text{if } n+1 \le i \le 2n. \end{cases} 
\end{equation}
Recall that
$\mathscr{P}^n_C$
is defined by 
\eqref{eq:PC}.
For any partition
$\la=(\la_1,\ldots,\la_l)$ in $\mathscr{P}^n_C$,
\begin{equation}
x_\la:=\rho_{\la_l}\cdots \rho_{\la_1}
\end{equation}
is an element of $W_\af^0$, and in fact 
$W_\af^0$ is bijective to $\mathscr{P}^n_C$ by this correspondence ({\cite[Lemma~24]{EE}}, {\cite[Lemma~5.3]{LSS}}).
Each element in $W_\af^0$ is called a \emph{$0$-Grassmannian element}.

\subsection{\texorpdfstring{Decomposition of $0$-Grassmannian elements}{Decomposition of 0-Grassmannian elements}}

Let $\{\eps_{1}, \ldots, \eps_{n}\}$ be the standard basis of the lattice $X^\vee$ of the cocharacters of $T\subset G$. 
The simple reflections $s_i,\;1\le i\le n,$
acts on $X^\vee$ by 
\begin{align}
s_i(\eps_j)&=\begin{cases}
\eps_{i+1} & \text{if $j=i$}\\
\eps_{i} & \text{if $j=i+1$}\\
\eps_j & \text{if $j\notin \{i,i+1\},$}
\end{cases}
\quad \text{for $1\le i\le n-1$, and}\\
s_n(\eps_j)&=\begin{cases}
-\eps_{n} & \text{if $j=n$}\\
\eps_{j} & \text{if $j\ne n$}
\end{cases}.
\end{align}
Recall that 
$W_\af\cong W\ltimes Q^\vee$
with 
$Q^\vee=\bigoplus_{i=1}^n\BZ \alv_i(\subset X^\vee)$ 
the coroot lattice of $G$. 
In our setting, we have 
\begin{equation}
\alv_i=\eps_i-\eps_{i+1} \text{ for } 1\le i\le n-1,\quad 
\alv_n=\eps_n.
\end{equation}
The fundamental coweights $\{\vpi_{i}^{\vee} \mid 1 \le i \le n\}$ of $G$ are given by 
\begin{equation}
\vpiv_i=\eps_1+\cdots+\eps_i \text{ for } 1 \le i \le n-1,\quad 
\vpiv_n=\frac{1}{2}
(\eps_1+\cdots+\eps_n).
\end{equation}
Note that $2\vpiv_n\in Q^\vee.$ 

For $\lambda \in \PC$, we define $v(\lambda) \in W$ and $\xi(\lambda) \in Q^{\vee}$ by 
\begin{equation}
x_{\lambda} = v(\lambda)t_{-\xi(\lambda)}. 
\end{equation}
In order to describe
$v(\lambda)$ and $\xi(\lambda)$, 
we define \begin{equation}
v_{i} := \begin{cases} s_{i} \cdots s_{n-1}s_{n}s_{n-1} \cdots s_{2}s_{1} & \text{if } 1 \le i \le n, \\ s_{2n-i} \cdots s_{2}s_{1} & \text{if } n+1 \le i \le 2n-1, \end{cases}\label{eq:defvi}
\end{equation}
and $v_{2n}=1.$
Note that $v_i=\rho_{i}\rho_{2n}^{-1}$,
and $\ell(v_i)=2n-i.$

\begin{lem}\label{lem:vrho}
$v_{i}^{-1}\rho_{i}=t_{-\eps_1}$
for any  $1\le i\le n.$
\begin{proof}
  This is verified by using 
$s_{0} = s_{\theta}t_{-\theta^{\vee}}$
and $s_{\theta} = s_{1} \cdots s_{n-1}s_{n}s_{n-1} \cdots s_{1}$, where  $\theta$ is the highest root of $G$ and 
$\theta^\vee=\eps_1$ is the corresponding coroot.  
\end{proof}
\end{lem}

\begin{prop} \label{prop:finite_part}
For $\lambda = (\lambda_{1}, \ldots, \lambda_{l}) \in \PC$, we have
\begin{equation}
v(\lambda) = v_{\lambda_{l}} \cdots v_{\lambda_{2}} v_{\lambda_{1}}, 
\end{equation}
and
\begin{equation}
    \xi(\lambda) = \eps_1 + v_{\lambda_{1}}^{-1}\eps_1 + v_{\lambda_{1}}^{-1} v_{\lambda_{2}}^{-1} \eps_1 + \cdots + v_{\lambda_{1}}^{-1} \cdots v_{\lambda_{l-1}}^{-1} \eps_1.
    \label{eq:zeta(la)}
\end{equation}
\end{prop}

\begin{proof}
We show the lemma by the induction on $l = \ell(\lambda)$. If $l = 1$, we have 
\begin{equation}
    x_{\lambda} = \rho_{\lambda_{1}} = v_{\lambda_{1}} (v_{\lambda_{1}}^{-1}\rho_{\lambda_{1}}) = v_{\lambda_{1}}t_{-\eps_1}
\end{equation}
by Lemma \ref{lem:vrho}.
Hence $v(\lambda) = v_{\lambda_{1}}$ and $\xi(\lambda) = \eps_1$, as desired. 

Let $l \ge 2$ and assume that the lemma holds for 
all the elements in 
$\PC$ with length less than or equal to $ l-1.$
Suppose $\lambda \in \PC$ has length $l$. 
By the inductive hypothesis, we have 
\begin{align}
    v(\lambda^{-}) &= v_{\lambda_{l}} \cdots v_{\lambda_{2}}, \\ 
    \xi(\lambda^{-}) &= \eps_1 + v_{\lambda_{2}}^{-1}\eps_1 + \cdots + v_{\lambda_{2}}^{-1} \cdots v_{\lambda_{l-1}}^{-1}\eps_1. 
\end{align}
We calculate that 
\begin{equation}
    x_{\lambda} 
    = x_{\lambda^{-}} \rho_{\lambda_{1}} 
    = (v(\lambda^{-})t_{-\xi(\lambda^{-})}) (v_{\lambda_{1}}t_{-\eps_1})
    = v(\lambda^{-})v_{\lambda_{1}} t_{-v_{\lambda_{1}}^{-1}\xi(\lambda^{-})} t_{-\eps_1
    }. 
\end{equation}
Hence we have 
\begin{align}
    v(\lambda) &= v(\lambda^{-})v_{\lambda_{1}} = v_{\lambda_{l}} \cdots v_{\lambda_{1}}, \\ 
    \xi(\lambda) &= \eps_1 + v_{\lambda_{1}}^{-1} \xi(\lambda^{-}) = \eps_1 + v_{\lambda_{1}}^{-1}\eps_1 + \cdots + v_{\lambda_{1}}^{-1} \cdots v_{\lambda_{l-1}}^{-1}\eps_1, 
\end{align}
as desired. 
\end{proof}

\subsection{\texorpdfstring{Determination of $[\xi(\la)]$ for $\la\in \PC$}{Determination of [xi(lambda)] for lambda in PnC}}We prove the following proposition.
\begin{prop}\label{prop:translation_explicit}
    For $\lambda \in \PC$, we have $[\xi(\lambda)] = \ell(\lambda)\alpha_{n}^{\vee}$. 
\end{prop}

In the following lemmas, set $\vpi_{0}^{\vee} := 0$, and the index $i$ of $\eps_i$
is taken modulo $n.$
\begin{lem}\label{lem:translation_n+1}
Assume that $\lambda = (\lambda_{1}, \ldots, \lambda_{l}) \in \PC$ satisfies $\lambda_{1} \le n+1$. Let $s, r \in \BZ$ be such that $l = sn + r$ and $0 \le r \le n-1$. Then
\begin{enumerate}[{\upshape {(}1{)}}]
\item $
\xi(\lambda) 
=2s\vpi_{n}^{\vee} + \vpi_{r}^{\vee},$
\item $\xi(\lambda) - \xi(\lambda^{-}) = \eps_{r}$. 
\end{enumerate}
\end{lem}
\begin{proof}
\noindent (1) 
We will show that
\begin{equation}
v_{\lambda_{1}}^{-1} \cdots v_{\lambda_{i}}^{-1} \eps_{1} = \eps_{i+1}\label{eq:epep}
\end{equation}
for $1\le i\le l-1$. Then 
it follows that from \eqref{eq:epep}  \begin{align}
 \xi(\lambda) 
&=\eps_1+\sum_{i=1}^{l-1}
v_{\la_1}^{-1}\cdots v_{\la_i}^{-1}(\eps_1) & & \text{by Proposition \ref{prop:finite_part}}\\
&=\sum_{i=0}^{l-1}\eps_{i+1} & & \text{by \eqref{eq:epep}}\\
&= s(\eps_{1} + \cdots + \eps_{n}) + (\eps_{1} + \cdots + \eps_{r}) \\ 
 &= 2s \vpi_{n}^{\vee} + \vpi_{r}^{\vee}, 
\end{align}
as desired. 

Now let us show \eqref{eq:epep}. 
It is easy to verify that if $1\le i< p\le n+1$, then we have
\begin{equation}
v_{p}^{-1}\eps_i=\eps_{i+1}.
\label{eq:vpep}
\end{equation}
Let $a$ be the maximal non-negative integer such that $\lambda_{a} = n+1.$ 
Then because 
$\lambda_{a+1}>\cdots>\lambda_i>1,$ we can apply \eqref{eq:vpep} to obtain
\begin{equation}
v_{\lambda_{a+1}}^{-1} \cdots v_{\lambda_{i}}^{-1} \eps_{1} = \eps_{i-a+1}.
\end{equation}
Now note that $i-a+1 \le n$ 
holds because $\lambda_{a+1}>\cdots>\lambda_i>1.$
Then, by applying equation~\eqref{eq:vpep} again, we have 
\begin{equation}
v_{\lambda_{1}}^{-1} \cdots v_{\lambda_{a}}^{-1} v_{\lambda_{a+1}}^{-1} \cdots v_{\lambda_{i}}^{-1} \eps_{1} = v_{n+1}^{-a} \eps_{i-a+1} = \eps_{i+1},\label{eq:xi(la)}
\end{equation}
as desired.

(2) follows from (1). 
\end{proof}

\begin{lem} \label{lem:translation_n+2}
Assume $\la\in \mathscr{P}_C^n$ satisfies $\la_1\ge n+2.$ Then there exists $1\le i\le n-1$ such that 
\begin{equation}
\xi(\la)-\xi(\la^{-})=\eps_i.
\end{equation}
\end{lem}
\begin{proof}
    Let $a$ be the maximal positive integer such that $\lambda_{a} \ge n+2.$ 
Let $s,r \in \BZ$ be such that $l-a =sn+r$ such that $0\le r \le n-1$. Define
\begin{equation}
l_i=2n-\la_i.
\end{equation} 
We will show the following 
by induction on $a$:
\begin{enumerate}[{\upshape {(}i{)}}]
\item There exists $t\ge 1$ and $1 \le i_{1} \le \cdots \le i_{t} \le n-1$ such that \begin{equation}
    \xi(\lambda) = 2s\vpi_{n}^{\vee} + \vpi_{i_{1}}^{\vee} + \cdots + \vpi_{i_{t}}^{\vee}.
\end{equation} 
\item If $t \ge 2$, then $i_2>l_{1}$. 
\item $\xi(\lambda) - \xi(\lambda^{-}) = \eps_{i_{1}}$. 
\end{enumerate}
Then the lemma follows. 

It is straightforward to show that
for $1\le i\le n,$ 
\begin{equation}
v_{\lambda_{1}}^{-1} (\eps_{1} + \cdots + \eps_{i}) = \begin{cases} \eps_{2} + \cdots + \eps_{i+1} & \text{if } i\le l_{1}, \\ \eps_{1} + \cdots + \eps_{i} & \text{if } i> l_{1}. \end{cases}
\label{eq:veps}
\end{equation}

We first assume $a = 1$. Then 
$\ell(\lambda^{-})=l-1=sn+r$.
By Lemma~\ref{lem:translation_n+1} applied to $\la^{-}$, we have $\xi(\lambda^{-}) = 2s\vpi_{n}^{\vee} + \vpi_{r}^{\vee}$. 
Hence 
we have 
\begin{align}
\xi(\lambda) &=
\eps_1 + v_{\la_1}^{-1}(\xi(\la^{-}))\\
&= \eps_1 + v_{\lambda_{1}}^{-1}(2s \vpi_{n}^{\vee} + \vpi_{r}^{\vee} )\\ 
&= \begin{cases} \eps_1 + 2s\vpi_{n}^{\vee} + (\eps_{2} + \cdots + \eps_{r+1}) & \text{if } r\le l_{1}  \\ \eps_1 + 2s\vpi_{n}^{\vee} + (\eps_{1} + \cdots + \eps_{r}) & \text{if } r>l_{1}  \end{cases} \\ 
&= \begin{cases} 2s\vpi_{n}^{\vee} + \vpi_{r+1}^{\vee} & \text{if } r\le l_{1}  \\ 2s\vpi_{n}^{\vee} + \vpi_{1}^{\vee} + \vpi_{r}^{\vee} & \text{if } r>l_{1} 
\end{cases}
\label{eq:zeta1}
\end{align}
by \eqref{eq:veps}, 
and therefore
\begin{equation}
\xi(\la)-\xi(\la^{-})=
\begin{cases}
\eps_{r+1} &\text{if $r\le l_1$}\\
\eps_1 &\text{if $r>l_1$}.
\end{cases}
\end{equation}
Note if $r\le l_1$
 then $r+1\le n-1$, because $\lambda_1\ge n+2.$
Thus we have (i) and (iii).
Regarding (ii), we only need to consider the case when 
$r>l_1$. Then the required inequality in (ii) is nothing but $r>l_1$.

Now assume $a\ge 2$.
By inductive hypothesis,
we can write for some $t\ge 1$, $\xi(\lambda^{-}) = 2s\vpi_{n}^{\vee} + \vpi_{j_{1}}^{\vee} + \cdots + \vpi_{j_{t}}^{\vee}$ 
with $1 \le j_{1} \le \cdots \le j_{t} \le n-1$, so that $j_2>l_{2} $ if $t \ge 2$
(note that the first component of $\lambda^{-}$
is $\lambda_{2}$). 

If $t\ge 2$, 
we claim that for $i\ge 2$,
\begin{equation}
v_{\lambda_{1}}^{-1} \vpi_{j_{i}}^{\vee}
= \vpi_{j_{i}}^{\vee}. \label{eq:kge2}
\end{equation} 
In fact, we have $l_{1} \le l_2< j_{2} \le j_{i}$,
therefore by the second case of \eqref{eq:veps}, we have \eqref{eq:kge2}.

Thus, by using \eqref{eq:veps} and \eqref{eq:kge2}, we obtain 
\begin{align}
\xi(\lambda) & = \eps_{1} + v_{\lambda_{1}}^{-1}(2s\vpiv_n+ \vpi_{j_{1}}^{\vee} + \cdots + \vpi_{j_{t}}^{\vee})\\
&=\begin{cases} \eps_{1} + 2s\vpi_{n}^{\vee} + (\eps_{2} + \cdots + \eps_{j_{1}+1}) + \vpi_{j_{2}}^{\vee} + \cdots + \vpi_{j_{t}}^{\vee} & \text{if $j_1\le l_{1}$} \\ \eps_{1} + 2s\vpi_{n}^{\vee} + (\eps_{1} + \cdots + \eps_{j_{1}}) + \vpi_{j_{2}}^{\vee} + \cdots + \vpi_{j_{t}}^{\vee} & \text{if $j_1>l_{1}$}  \end{cases} \\ 
&= \begin{cases} 2s\vpi_{n}^{\vee} + \vpi_{j_{1}+1}^{\vee} + \vpi_{j_{2}}^{\vee} + \cdots + \vpi_{j_{t}}^{\vee} & \text{if $j_1\le l_{1}$}  \\ 2s\vpi_{n}^{\vee} + \vpi_{1}^{\vee} + \vpi_{j_{1}}^{\vee} + \vpi_{j_{2}}^{\vee} + \cdots + \vpi_{j_{t}}^{\vee} & \text{if $j_1>l_{1}$}, \end{cases}
\label{eq:zeta2}
\end{align}
and hence
\begin{equation}
    \xi(\lambda) - \xi(\lambda^{-}) = \begin{cases}
        \eps_{j_{1}+1} & \text{if $j_1\le l_{1}$} \\ 
        \eps_{1} & \text{if $j_1>l_{1}$}.
    \end{cases}\label{eq:diffxi}
\end{equation}
Note that \eqref{eq:diffxi}
holds also in the case $t=1$ by the same computation.

If $j_1\le l_{1} $, then 
we have $j_1+1\le j_2$ because $l_1\le l_2$
and $l_2< j_2.$
Therefore  (i) is satisfied. By \eqref{eq:diffxi}, (iii) also holds. In this case, condition (ii) also holds because $l_{1}\le l_2<j_2$. If $j_1>l_{1}$, then (i) and (iii) are obviously satisfied. In this case, $j_1>l_{1}$ is exactly the condition (ii). This completes the proof. 
\end{proof}

\begin{proof}[Proof of Proposition \ref{prop:translation_explicit}] Note that for $1\le i\le n$,
\begin{equation}  \eps_i=\alpha_i^{\vee}+\cdots+\alpha_n^{\vee}
\equiv \alpha_n^\vee \mod Q_P^\vee.
\label{eq:epsi_mod}
\end{equation}
Therefore from Lemma~\ref{lem:translation_n+1}\,(2) and Lemma~\ref{lem:translation_n+2}, we have
\begin{equation}
[\xi(\lambda)]-[\xi(\lambda^{-})]
=\alpha_n^\vee.
\end{equation}
We use this equation 
repeatedly to obtain the proposition.
\end{proof}

\subsection{\texorpdfstring{Determination of $\mcr{v(\lambda)}$ for $\la\in \SP(n)$}{Determination of [v(lambda)] for lambda in SP(n)}} 
\label{sec:partition_correspondence}
The Weyl group is $W=\langle s_1,\ldots,s_{n}\rangle.$
Consider the parabolic
subgroup $W_{P}=\langle s_1,\ldots,s_{n-1}\rangle$ of $W$. Let $W^P$ be the set of minimal length coset representatives for $W/W_P$.

There is a bijection $\SP(n) \rightarrow W^{P}$,
$\lambda\mapsto u_\lambda$, given as follows. 
For $1 \le k \le n$, we set 
\begin{equation}
    u_{k} := s_{n+1-k} \cdots s_{n-1} s_{n}. 
\end{equation}
Then one can check that $u_{k}\in W^P.$ More generally, 
for any strict partition $\lambda = (\lambda_{1}, \ldots, \lambda_{l}) \in \SP(n)$, we define $u_{\lambda} \in W$ by 
\begin{equation}
    u_{\lambda} := u_{\lambda_{l}} \cdots u_{\lambda_{1}}. 
\end{equation}
Then we have
\begin{equation}
W^P=\{u_\la\mid \la\in \SP(n)\}.
\end{equation}

For $1 \le k \le n$, set $k^* := n+1-k$. 
Then for $\lambda = (\lambda_{1}, \ldots, \lambda_{l}) \in \SP(n)$, we define $\lambda^* \in \SP(n)$ by 
\begin{equation}
    \lambda^* := (\lambda_{l}^*, \ldots, \lambda_{1}^*). 
\end{equation}
In this section, we show the following proposition. 
\begin{prop} \label{prop:partition_correspondence}
    For $\lambda \in \SP(n)$, we have $\mcr{v(\lambda)} = u_{\lambda^*}$. 
\end{prop}

In order to prove Proposition~\ref{prop:partition_correspondence}, we prepare the following lemmas which can be proved in a straightforward manner.  
\begin{lem} \label{lem:commutation_1}
    For $1 \le i \le j \le n-1$, we have 
    \begin{equation}
        u_{i^*}u_{j^*} = u_{(j+1)^*}u_{i^*}s_{n-1}. 
    \end{equation}
\end{lem}

Recall that $v_i$'s are defined by \eqref{eq:defvi}. 
\begin{lem} \label{lem:commutation_2}
    For $2 \le i \le n$, we have 
    \begin{equation}
        v_{n+1}u_{i^*} = u_{(i-1)^*}v_{n+2}. 
    \end{equation}
\end{lem}

\begin{proof}[Proof of Proposition~\ref{prop:partition_correspondence}]
    First, note that for $1 \le i \le n$, we have 
    \begin{equation}
        v_{i} = u_{i^*}v_{n+1}. \label{eq:vi}
    \end{equation}
    
    In order to prove the proposition, it suffices to show $v(\lambda) = u_{\lambda^*}v_{n+1}^{l}$
    because $v_{n+1} \in W_{P}$. 
    We use induction on $l = \ell(\lambda)$ to show this. 
    If $l = 1$, then we have 
    \begin{equation}
        v(\lambda) = v_{\lambda_{1}} = u_{\lambda_{1}^*}v_{n+1} = u_{\lambda^*}v_{n+1}, 
    \end{equation}
    as desired; here we used Proposition~\ref{prop:finite_part}
    in the first equality,
    and \eqref{eq:vi}
   in the second equality.

     Let $l > 1$ and assume that the proposition holds the elements in $\SP(n)$ such that the length is less than $l$. Let 
     $\lambda = (\lambda_{1}, \ldots, \lambda_{l}) \in \SP(n)$ be such that $\ell(\lambda) = l$. 
     Set $\mu := (\lambda_{1}, \ldots, \lambda_{l-1})$.
    Then 
    \begin{align}
        v(\lambda) 
        &= v_{\lambda_{l}} v(\mu) &\\ 
        &= (u_{\lambda_{l}^*}v_{n+1})(u_{\mu^*}v_{n+1}^{l-1}) & & \text{by \eqref{eq:vi} and induction}\\ 
        &= u_{\lambda_{l}^*}v_{n+1}u_{\lambda_{1}^*}u_{\lambda_{2}^*} \cdots u_{\lambda_{l-1}^*}v_{n+1}^{l-1} \\ 
        &= u_{\lambda_{l}^*}u_{(\lambda_{1}-1)^*}v_{n+2}u_{\lambda_{2}^*} \cdots u_{\lambda_{l-1}^*}v_{n+1}^{l-1} 
        & & \text{by Lemma \ref{lem:commutation_2}}\\ 
        &= u_{\lambda_{1}^*}u_{\lambda_{l}^*}s_{n-1}v_{n+2}u_{\lambda_{2}^*} \cdots u_{\lambda_{l-1}^*}v_{n+1}^{l-1} & & \text{by Lemma \ref{lem:commutation_1}}\\ 
        &= u_{\lambda_{1}^*}u_{\lambda_{l}^*}v_{n+1}u_{\lambda_{2}^*} \cdots u_{\lambda_{l-1}^*}v_{n+1}^{l-1}
        & & \text{since $s_{n-1}v_{n+1}=v_{n+1}$}\\ 
        &= \cdots \\ 
        &= u_{\lambda_{1}^*}u_{\lambda_{2}^*} \cdots u_{\lambda_{l}^*} v_{n+1}^{l} \\ 
        &= u_{\lambda^*}v_{n+1}^{l}, 
    \end{align}
    as desired.
\end{proof}

%=================%
% START SECTION 4 %
%=================%

\section{Proof of the main result} \label{sec:proof}
Recall that $K_*^T(\GrC)$ is an $R(T)$-module
with basis $\{\CO_\la^{\Gr} \mid \la\in \mathscr{P}^n_C\},$
and $\Pet: K_*^T(\GrC)_{\loc} \rightarrow
QK_T^{\poly}(\LG)_{\loc}$ is the map
given in Theorem  \ref{thm:parabolic_peterson_surj}. 
In this section, we set $Q = Q_{n}$. 

\begin{lem}\label{lem:Pet n+1} We have
  \begin{equation}
\Pet(\CO^\Gr_{(n+1)})=Q^{-1}.
\label{eq:Phi(n+1)}
\end{equation}  
\end{lem}
\begin{proof} By
Proposition \ref{prop:finite_part}, 
$x_{(n+1)}=v_{n+1}t_{-\eps_1}$.
Because $\mcr{v_{n+1}}=1$ and
$\eps_1\equiv \alpha_n^\vee\mod Q_P^\vee$ (cf. \eqref{eq:epsi_mod}),
we have
\begin{equation}
\Pet(\CO^\Gr_{(n+1)})
=Q^{[-\eps_1]}
\CO_{\LGrm}^{\mcr{v_{n+1}}}
=Q^{-1}.
\end{equation}
\end{proof}

\begin{lem}\label{lem:main}
If $\lambda \in \PC$ satisfies $\lambda_{1} \ge n+1$, then 
\begin{equation}
\CO^\Gr_{\lambda} - \CO^\Gr_{(n+1)} \CO^\Gr_{\lambda^{-}} \in \mathrm{Ker}(\Pet). 
\end{equation}
\end{lem}
\begin{proof}
We first note that $v_{\la_1}\in W_{P} $ because $\la_1\ge n+1$. Therefore we have  
$v(\lambda) = v(\lambda^{-})v_{\la_1}\in 
v(\lambda^{-})W_{P},$
which implies \begin{equation}\mcr{v(\lambda)}=\mcr{v(\lambda^{-})}.\label{eq:mcr}
\end{equation}
By Proposition \ref{prop:translation_explicit}, we have
\begin{equation}
Q^{[-\xi(\la)]}=Q^{-1}Q^{[-\xi(\la^{-})]}.
\end{equation}
Thus it follows that
\begin{equation}
\Pet(\CO^\Gr_{\lambda}) 
 = Q^{[-\xi(\lambda)]} \CO_{\LGrm}^{\mcr{v(\lambda)}} 
= Q^{-1} Q^{[-\xi(\lambda^{-})]} \CO_{\LGrm}^{\mcr{v(\lambda^{-})}} 
=Q^{-1}\Pet(\CO^\Gr_{\lambda^{-}}).
\end{equation}
On the other hand, from Lemma \ref{lem:Pet n+1},
\begin{align}
\Pet(\CO^\Gr_{(n+1)}\cdot \CO^\Gr_{\lambda^{-}})
=\Pet(\CO^\Gr_{(n+1)})\Pet(\CO^\Gr_{\lambda^{-}})
=Q^{-1}\Pet(\CO^\Gr_{\lambda^{-}}).
\label{eq:lambda}
\end{align}
Therefore the Lemma holds.
\end{proof}

\begin{proof}[Proof of Theorem~\ref{thm:A}]
By Lemma~\ref{lem:main}, the composition 
\begin{equation}
\KGrC\hookrightarrow
\KGrC_\loc
\xrightarrow{\Phi_{P}} 
QK_{T}^{\poly}(\LG)_\loc
\end{equation}
induces
an $R(T)$-algebra homomorphism
\begin{equation}
    \mapname{\overline{\Phi}_P}\map{\KGrC/\JLG}{QK_{T}^{\poly}(\LG)_\loc}{\CO_{\lambda}^{\Gr}}{Q^{-[\xi(\lambda)]}\CO_{\LGrm}^{\mcr{v(\lambda)}};}
\end{equation}
recall that $Q^{-[\xi(\lambda)]} = Q^{-\ell(\lambda)}$ by Proposition~\ref{prop:translation_explicit}.
Since $\overline{\Phi}_{P}(\CO_{(n+1)}^{\Gr}) = Q^{-1}$ (by equation~\eqref{eq:Phi(n+1)}) is invertible in $QK_{T}^{\poly}(\LG)_\loc$, $\overline{\Phi}_P$ induces an $R(T)$-algebra homomorphism
\begin{equation}
    \mapname{\overline{\Phi}_{P,\loc}} \map{(\KGrC/\JLG)_\loc}{QK_{T}^{\poly}(\LG)_\loc}{\CO_{\lambda}^{\Gr}(\CO_{(n+1)}^{\Gr})^{-k}}{Q^{k}Q^{-\ell(\lambda)}\CO_{\LGrm}^{\mcr{v(\lambda)}}},
\end{equation}
for $k\in \BZ.$
Moreover, $\overline{\Phi}_{P}(\CO_{(n+1)}^{\Gr}) = Q^{-1}$ implies that $\overline{\Phi}_{P, \loc}$ is an $R(T)[Q, Q^{-1}]$-algebra homomorphism. 

We prove bijectivity of $\overline{\Phi}_{P, \loc}$. 
By the explicit form 
of $\JLG$, one sees that 
$(\KGrC/J_{\LG})_{\loc}$, 
as an 
$R(T)[Q,Q^{-1}]$-module, is generated 
by $\CO_{\lambda}^\Gr$ with $\la \in \SP(n)\subset \PC$. 
Take $\lambda \in \SP(n)$. 
By Proposition~\ref{prop:partition_correspondence}, we have $\CO_{\LGrm}^{\mcr{v(\lambda)}} = \CO_{\LGrm}^{\lambda^*}$. 
Hence 
\begin{equation}
    \overline{\Phi}_{P,\loc}(\CO_{\lambda}^{\Gr}(\CO_{(n+1)}^{\Gr})^{-\ell(\lambda)}) = Q^{\ell(\lambda)} Q^{-\ell(\lambda)} \CO^{\mcr{v(\lambda)}}_{\LGrm} = \CO^{\lambda^*}_{\LGrm}. 
\end{equation}
Recall that the set $\{\CO_{\LGrm}^{\lambda} \mid \lambda \in \SP(n) \}$ is a basis of $QK_{T}^{\poly}(\LG)_{\loc}$ as an $R(T)[Q,Q^{-1}]$-module. 
Noting that the assignment $\lambda \mapsto \lambda^*$ defines a bijection $\SP(n) \rightarrow \SP(n)$, 
we conclude that $\overline{\Phi}_{P, \loc}$ is bijective. 
\end{proof}

%=================%
% START SECTION 5 %
%=================%

\section{Limit to homology}
\label{sec:quantum_cohomology}
In this section, we discuss
the limit of our construction to
equivariant homology case.
As a biproduct, we give an explicit 
description of the set of Peterson's coset representatives for $\LG$.
\subsection{\texorpdfstring{Filtration on $K_*^T(\Gr_G)$}{Filtration on K*T(GrG)}}
We regard the ring 
${H_T^*(pt)}=\mathbb{C}[\eps_1,\ldots,\eps_n]$
as a graded ring. 
Let $\hat{S}=\mathbb{C}\bra{\eps_1,\ldots,\eps_n}$
be the completion of $S$ 
by the ideal $\mathfrak{m}=(\eps_1,\ldots,\eps_n).$
The ring $\hat{S}$ is a filtered ring equipped with a
decreasing filtration 
$\hat{S}\supset \mathfrak{m}\supset \mathfrak{m}^2\supset \cdots.$
The associated graded ring $\mathrm{gr}(\hat{S})$ is naturally isomorphic to $S.$
We identify $R(T)$ as
a subring of $\hat{S}$ by sending $e^{\eps_i}$ to 
$\sum_{n = 0}^{\infty} \eps_{i}^{n}/n! \in \hat{S}$. 

Define $\hat{K}_*^T(\Gr_G)=\hat{S}\otimes_{R(T)}
K_*^T(\Gr_G).$
We regard the ring $\hat{K}_*^T(\Gr_G)$
as a filtered 
$\hat{S}$-algebra 
with a filtration
\[
F_0\hat{K}_*^T(\Gr_G)
\subset 
F_1\hat{K}_*^T(\Gr_G)
\subset
\cdots
\subset \hat{K}_*^T(\Gr_G)
,
\]
\[
F_i\hat{K}_*^T(\Gr_G)=
\sum_{\substack{w\in W_\af^0\\\ell(w)\le i}}
\hat{S}\CO_w^\Gr
+\sum_{\substack{w\in W_\af^0\\\ell(w)> i}}
\mathfrak{m}^{\ell(w)-i}\CO_w^\Gr.
\]
The associated graded ring
\[
\mathrm{gr}_F\left(\hat{K}_*^T(\Gr_G)\right)
=\bigoplus_{i\ge 0}
F_i\hat{K}_*^T(\Gr_G)/F_{i-1}\hat{K}_*^T(\Gr_G)
\]
is isomorphic, as a graded $S$-algebra, to $H_*^T(\Gr_G)$.
For $\la\in \PC$ with $|\la|=i,$
the image of 
$\CO^\Gr_\la\in F_i\hat{K}_*^T(\Gr_G)$
is $\sigma_\la^\Gr\in H_{2i}^T(\Gr_G).$

Let $\hat{J}_{\LG}
$
be an ideal
of $\hat{K}_*^T(\Gr_G)$ 
corresponding to \eqref{eq:def_JLG}.
The quotient ring 
$\hat{K}_*^T(\Gr_G)/\hat{J}_{\LG}$ inherits a 
filtration.
The associated graded ring
is given by 
\begin{align}
&\mathrm{gr}_F(\hat{K}_*^T(\Gr_G)/\hat{J}_{\LG})\\
 &=\bigoplus_{i\ge 0}
 F_i\left(\hat{K}_*^T(\Gr_G)/\hat{J}_{\LG}\right)/F_{i-1}\left(\hat{K}_*^T(\Gr_G)/\hat{J}_{\LG}\right)\\
&\cong \bigoplus_{i\ge 0}
F_{i}\hat{K}_*^T(\Gr_G)
/\left(F_{i}\hat{K}_*^T(\Gr_G)
\cap \hat{J}_{\LG}+F_{i-1}\hat{K}_*^T(\Gr_G)\right).
\end{align}

Recall that ${J^{\hom}_{\LG}}$ is 
the ideal of 
$H_*^T(\Gr_G)$ given
as the $S$-span of 
$\sigma_\la^\Gr$
such that $\la\in\PC$ and 
$\la_1\ge n+2.$

\begin{prop}\label{prop:gr-K-to-H}
We have an isomorphism
of graded $S$-algebras
\begin{equation}
\mathrm{gr}_F\left
(\hat{K}_*^T(\Gr_G)/\hat{J}_{\LG}
\right)
\cong
H_*^T(\Gr_G)
/{J^{\hom}_{\LG}}.\label{eq:grK}
\end{equation} 
For $\la\in \PC,$ with $|\la|=i,$ the class in  $F_i(\hat{K}_*^T(\Gr_G))$ determined by 
$\CO^\Gr_\la$ 
is identified with $\sigma_\la^\Gr \mod J_{\LG}^{\mathrm{hom}}\in H_i^{T}(\Gr_G)/J_{\LG}^{\mathrm{hom}}$. 
\end{prop}
\begin{proof}
Consider the homomorphism of 
graded $S$-algebras:
\[
\varphi_{\LG}:
H_*^T(\Gr_G)
\xrightarrow{\sim}
\mathrm{gr}_F\left
(\hat{K}_*^T(\Gr_G)
\right)
\rightarrow\mathrm{gr}_F\left
(\hat{K}_*^T(\Gr_G)/\hat{J}_{\LG}
\right)
\]
Clearly $\varphi_{\LG}$ is surjective. We will show that 
$\kernel(\varphi_{\LG})=J^{\hom}_{\LG}.$

For $\la\in \PC$ with $|\la|=i$, 
we denote the image of
$\CO^\Gr_\la$
under the map
\[
F_i
\hat{K}_*^T(\Gr_G)\rightarrow
F_i\left(\hat{K}_*^T(\Gr_G)/\hat{J}_{\LG}\right)
\]
by $\pi_i(\CO^\Gr_\la).$
If 
$\la_1\ge n+2,$
we have
\begin{align}
\CO^\Gr_\la&=
(\CO^\Gr_\la-
\CO^\Gr_{(n+1)}
\CO^\Gr_{\la^{-}})
+\CO^\Gr_{(n+1)}
\CO^\Gr_{\la^{-}}\\
&\in (F_{i}\hat{K}_*^T(\Gr_G))
\cap \hat{J}_{\LG}
+F_{i-1}\hat{K}_*^T(\Gr_G).
\end{align}
Thus, in this case, we have 
$\pi_i(\CO^\Gr_\la)=0
$ and therefore $\varphi_{\LG}(\sigma_\la^\Gr)=0.$

Next suppose $\la_1\le n+1$.
We write $\la=((n+1)^a,\la^{\le n })$, then 
\begin{equation}
\CO^\Gr_\la\equiv (\CO_{(n+1)}^\Gr)^a\CO^\Gr_{\la^{\le n }}
\mod \hat{J}_{\LG}.
\end{equation}
We claim that
$\pi_i(\CO_\la^\Gr)$ is not zero. In fact, it is easy to see
\[
\hat{K}_*^T(\Gr_G)/\hat{J}_{\LG}
\cong
\bigoplus_{i\ge 0,\; \la\in \SP(n)}
\hat{S}\,\overline{\CO^\Gr_\la} \cdot (\overline{\CO^\Gr_{(n+1)}})^i.
\]
Then the claim follows.
Hence 
$\sigma_\la^\Gr$ is not contained in $ 
\kernel(\varphi_{\LG}).$ 
The second statement is obvious.
\end{proof}

\begin{proof}[Proof of Corollary \ref{cor:homology}]
The equivariant quantum $K$-ring 
$QK_T(G/P)$ has a natural decreasing filtration
\[
QK_T(G/P)=
F^0QK_T(G/P)
\supset F^1QK_T(G/P)\supset\cdots
\]
 and its associated graded ring is
isomorphic, as an $H_T^*(pt)[q]$-algebra, to $QH_T^*(G/P)
$ (see for example \cite[\S 8]{GMSZ}).
The isomorphism of Theorem \ref{thm:A} induces a degree reversing isomorphism in the Corollary \ref{cor:homology}. 
By the proof of Proposition \ref{prop:gr-K-to-H}, we have \eqref{eq:homology-map}. 
\end{proof}

 \subsection{Peterson coset representatives}\label{ssec:Pet}
Let $L_P$ be the Levi part of $P.$
Let $W_P,R_P,Q_P^\vee,$ be 
the the Weyl group, 
the root system,   
and the coroot lattice
of $L_P$
respectively. 
Let $(W_P)_\af=W_P\ltimes Q_P^\vee$. 
There is a distinguished 
subset $(W^P)_\af\subset W_\af$
such that the map
$(W^P)_\af \times (W_P)_\af
\rightarrow 
W_\af, (w_1,w_2)\mapsto w_1w_2$ is bijective. 
We call $(W^P)_\af$ 
the set of Peterson's coset representatives.
We refer to \cite[\S 10.3]{LS} for the basic properties of $(W^P)_\af.$
Let $R_P^+$ be the positive system of $R_P.$
An element 
$x=wt_\xi\in W_\af\; (w\in W, \xi\in Q^\vee)$ is
in $(W^P)_\af$
is and only if
\[
\langle \xi,\alpha\rangle=\begin{cases}
0 & \text{if $w(\alpha)>0$}\\
-1 & \text{if $w(\alpha)<0$}
\end{cases}
\]
for all $\alpha\in R_P^+$ (\cite[Lemma 10.2]{LS}).
Let
\[
J_P^{\mathrm{Pet}}
=\bigoplus_{x\in W_\af^0\setminus (W^P)_\af}
S \sigma_{x}^\Gr.
\]

Let us denote the ideal $J^{\mathrm{Pet}}_{P}$
for $G/P=\LG$
by 
$J^{\mathrm{Pet}}_{\LG}.$
Then
the parabolic Peterson isomorphism 
\cite[Theorem 10.21]{LS} reads
\begin{equation}
 \left(H_*^T(\Gr_G)
/J_{\mathrm{LG}(n)}^{\mathrm{Pet}}\right)_\loc 
\cong QH_T^*(\LG)_\loc
\label{eq:hom_Para_Pet}
\end{equation}
 in our case.
\begin{rem}\label{rem:inj}
As a consequence of \eqref{eq:hom_Para_Pet}, we can regard
$ \left(H_*^T(\Gr_G)
/J_{\mathrm{LG}(n)}^{\mathrm{Pet}}\right)_\loc$
as a free $\BZ[q,q^{-1}]$-module
such that $q^{-1}$ acts as 
multiplication by $\sigma_{(n+1)}^\Gr.$
In particular, 
the natural map
\[H_*^T(\Gr_G)
/J_{\mathrm{LG}(n)}^{\mathrm{Pet}}
\rightarrow
\left(H_*^T(\Gr_G)
/J_{\mathrm{LG}(n)}^{\mathrm{Pet}}\right)_\loc
\]
is injective.
From Corollary \ref{cor:homology},
the similar statement holds
for $J_{\mathrm{LG}(n)}^{\mathrm{hom}}$
instead of $J_{\mathrm{LG}(n)}^{\mathrm{Pet}}$. 
\end{rem}
\begin{cor}\label{cor:Petco} 
Via the bijection
$W_\af^0\cong \PC$, the subset 
$W_\af^0\cap (W^P)_\af$ corresponds to the 
set consisting of $\la\in\PC$ such that 
$\la_1\le n+1.$
\end{cor}
\begin{proof}
For $\la\in \PC$, let 
$\alpha_\lambda$ and $\beta_\la$ denote
the images of 
$\sigma_\la^\Gr$
in $\left(H_*^T(\Gr_G)
/J_{\mathrm{LG}(n)}^{\mathrm{Pet}}\right)_\loc$
and $\left(H_*^T(\Gr_G)
/J_{\mathrm{LG}(n)}^{\mathrm{hom}}\right)_\loc$
respectively. 
 By Remark \ref{rem:inj}, we have  
\[
\alpha_\la \ne 0
\Longleftrightarrow
x_\la\in (W^P)_\af.
\]
Similarly, we have
\[
\beta_\la \ne 0
\Longleftrightarrow
\la_1\le n+1.
\]
Note that 
$\alpha_\lambda$ and $\beta_\la$
 correspond to the same element of 
 $QH_T^*(\LG)_\loc$ via
 \eqref{eq:hom_Para_Pet} and 
 Corollary \ref{cor:homology} respectively.
Hence 
the corollary holds.
\end{proof}

\end{document}